\documentclass{article}

\usepackage[english]{babel}
\usepackage{amssymb}
\usepackage{amsthm}
\usepackage{arxiv}

\newtheorem{proposition}{Proposition}
\newtheorem{theorem}{Theorem}
\newtheorem{remark}{Remark}
\newcommand{\Iso}{\ensuremath{\mathit{Iso}}}

\newcommand \n {\{ 1, 2, \ldots , n \}}
\newcommand \m {\{ 1, 2, \ldots , m \}}

\begin{document}

\title{Isometry invariant permutation codes and mutually orthogonal Latin squares}

\author{
Ingo Janiszczak\thanks{ingo@iem.uni-due.de} \ and Reiner Staszewski\\
Faculty of Mathematics\\
University of Duisburg-Essen\\
45127 Essen, Germany\\
}

\maketitle

\begin{abstract}
{
Commonly the direct construction and the description of mutually orthogonal Latin
squares (MOLS) makes use of difference or quasi-difference matrices.
Now there exists a correspondence between MOLS and separable permutation 
codes. We like to present separable permutation codes of length $35$, $48$, $63$ and $96$
and minimum distance $34$, $47$, $62$ and $95$ consisting of $6 \times 35$, $10 \times 48$, 
$8 \times 63$ and $8 \times 96$ codewords respectively. Using the correspondence this gives $6$ MOLS
for $n=35$, $10$ MOLS for $n=48$, $8$ MOLS for $n=63$ and $8$ MOLS for $n=96$. 
So $N(35) \ge 6$, $N(48) \ge 10$, $N(63) \ge 8$ and $N(96) \ge 8$ holds which are new lower 
bounds for MOLS.
The codes will be given by generators of an appropriate subgroup $U$ of the isometry group of the symmetric 
group $S_n$ and $U$-orbit representatives.
This gives an alternative uniform way to describe the MOLS where the data for the codes 
can be used as input for computer algebra systems like MAGMA, GAP etc.
}
\end{abstract}


\section{Introduction}
\label{sect:intro}

Let $n$ be an integer and $V$ and $W$ be sets consisting of $n$ elements.
A Graeco-Latin square is a $n \times n$ matrix 
$M = ((v_{ij},w_{ij}))$ with entries in 
the Cartesian product $V \times W$ such that the set of all different 
entries in $M$ equals $V \times W$ and the matrices $L_1 := (v_{ij})$ and
$L_2 := (w_{ij})$ are Latin squares, i.e. all rows and all columns of $L_1$
and $L_2$ are permutations of $V$ and $W$ respectively. On the other hand
two such Latin squares are called orthogonal if the matrix $((v_{ij},w_{ij}))$
is a Graeco-Latin square. A set of mutually orthogonal Latin squares (MOLS) 
consists of Latin squares which are pairwise orthogonal.
Let $N(n)$ be the largest size for MOLS of order $n$. In 1960 R.C. Bose,
S.S Shrikhande and E.T. Parker \cite{BSP60} proved the existence of 
a Graeco-Latin square (which is equivalent to $N(n) \ge 2$) for all
integers $n \ge 3$, $n \ne 6$ which disproved a conjecture from L. Euler
in 1782 \cite{Euler1782}. It is easy to that $N(n) \le n-1$ and equality
holds if $n$ is a prime power but for all other integers $n \ge 10$ the
number $N(n)$ is not known and there just exist lower bounds.

In the last decades improving a lower bound for $N(n)$ was always obtained
either by construction of a $(G,m+1;\lambda)$-difference matrix a
$(G,m+2;\lambda,\mu;u)$-quasi-difference matrix for a finite group $G$ or recursive constructions.
For the definitions of these kind of matrices see $17.1$ and $17.44$ in \cite{Handbook07}.

In this paper we like to proceed in a different way. For explanation let us 
recall the notation in \cite{JLOS15} and some definitions from there.

Let $n$ be a positive integer and let $X$ be an arbitrary finite set of order $n$.
Let $S_X := \{ \sigma : X \rightarrow X \mid \sigma \;\mbox{is bijective} \}.$
This group acts on $X$ from right and
for $x \in X, \sigma \in S_X,$ we denote the image of $x$ under $\sigma$ by $x^{\sigma}.$
If $X = \n$ we will write $S_n$ instead of $S_X.$
For $\sigma, \tau \in S_X$ the Hamming distance $d_H(\sigma, \tau) = d(\sigma, \tau)$ equals $n$ minus 
the number of fixed points of $\sigma^{-1} \tau$,
and for $S, T \subseteq S_X$ let $d(S, T) := \min \ \{ d(\sigma, \tau) \mid \sigma \in S, \tau \in T \}.$
A subset $C$ of $S_X$ is called a \emph{permutation code} 
or \emph{permutation array} of length $n$ and of minimum distance
$$d(C) := \min \ \{d_H(\sigma, \tau) \mid \sigma, \tau \in C , \sigma \ne \tau \} \ .$$
For brevity such a permutation code $C$ is called an $(n,d)$-PA.
An $(n,n-1)$-PA $C$ is called $(r,m)$-separable if it is the join of $m$ disjoint $(n,n)$-PAs 
$L_1, \ldots, L_m$ of cardinality $r$ such that 
for all pairs $L_i, L_j$, $i \ne j$ the distance $d_H(\sigma, \tau)$ equals $n-1$ for 
all $\sigma \in L_i$ and all $\tau \in L_j$. 
For $r = n$ the code $C$ corresponds to $m$ MOLS \cite{CKL04}.

$S_n$ is a metric space via the Hamming distance and the isometry group $\Iso(n)$ 
is isomorphic to the wreath product $S_n \wr S_2$ \cite{F60}.
$\Iso(n)$ can be described as subgroup of $S_{2n}$.
Let $B_1$ and $B_2$ be the naturally embedded subgroups isomorphic to $S_n$
acting on the sets $\{1,2,\ldots ,n\}$ and $\{n+1,n+2,\ldots ,2n\}$, respectively. 
Moreover let $t_n := (1,n+1)(2,n+2)\cdots (n,2n) \in S_{2n}$. Then 
$\Iso(n) = \langle B_1, B_2, t_n\rangle = (B_1 \times B_2):\langle t_n \rangle$,
and the action of this group on $B_1$ from right is given by
$$b * x := \left\{ \begin{array}{lll}
x^{-1} b \; & \mbox{if} \; & x \in B_1 \\
b \varphi(x) \; & \mbox{if} \; & x \in B_2 \\
b^{-1} \; & \mbox{if} \; & x = t_n, \\
\end{array}
\right. $$
where $b \in B_1$, $x \in Iso(n)$ and $\varphi$ denotes the natural isomorphism from $B_2$ to $B_1$.
Moreover, if $U$ is a subgroup of $\Iso(n)$ and $b \in B_1,$
$b * U = \{ b * u \mid u \in U \}$ denotes the $U-$orbit of $b.$
Our codes $C$ will now be regarded as subsets of $B_1$ and a code $C$ will be called 
$U-$invariant if
$C$ is closed under the action of $U$, which means that $C$ is a union of $U-$orbits.
The strategy of constructing $(n,m)$-separable PAs invariant under a given subgroup $U$ of
$\Iso(n)$ we like to refer to \cite{JLOS15}. But now we will join only separable $U$-orbits.

In section~\ref{sect:impr}  we will proof the existence of $6$ MOLS for $n=35$, 
$10$ MOLS for $n=48$, $8$ MOLS for $n=63$ and  
$n=96$ by constructing a $(35,6)$ - separable, a $(48,10)$ - separable, a $(63,8)$ - separable and a 
$(96,8)$ - separable PA respectively. These are invariant under specific subgroups 
$U$ of $Iso(n)$ using the described group action of $Iso(n)$ on $S_n$ and will be
given by generators of $U$ and representatives of the $U$-orbits. 

Furthermore we like to introduce a correspondence between
$(n,m+2,1)$-difference matrices and $(n,m)$-separable PAs with special kind 
of isometry groups in section~\ref{sect:diffmatcoset}

In our studies we also investigated cases different from $n = 35$, $n = 48$, $n = 63$ and $n = 96.$
Some interesting results are summarized in section~\ref{sect:knownresults}

Since all data for the constructed PAs are contained in the symmetric group $S_{2n}$ 
it can easily be read into a computer algebra system like MAGMA \cite{BCP97} or GAP \cite{GAP}.
This is a uniform and compact way to describe $m$ MOLS of order $n$ with natural algebraic objects.

\section{Difference matrices and cosets}
\label{sect:diffmatcoset}

In this section we will prove that the concept of difference matrices over a group $G$ 
agrees with the concept of separable permutation codes
having an isometry group which contains a special group which is naturally isomorphic to $G.$

Let $G = \{ g_1 = 1_G, g_2, \ldots g_n \}.$
For each $j \in \n$ let $\gamma_j : G \rightarrow G$ be the bijection defined by
$g^{\gamma_j} := g g_j.$
Now
$g^{\gamma_i \gamma_j} = \left(g^{\gamma_i}\right)^{\gamma_j} = (g g_i) g_j = g (g_i g_j)$
shows that mapping $g_i$ to $\gamma_i$ is an isomorphism between $G$ and 
${\cal{R}}(G) := \{ \gamma_1 = Id, \gamma_2, \ldots \gamma_n \} \le S_G.$
${\cal{R}}(G)$ is called the (right-)regular representation of $G$.
\bigskip
\begin{proposition}
\label{prop1}
Let $G$ be a finite group of order $n$. A $(G, m + 1; 1)-$difference matrix D exists if and only if there exists a set $\{ \theta_1 = Id, \theta_2, \ldots, \theta_{m} \} \subseteq S_G$ such that
$d({\cal{R}}(G) \theta_{i'}, \theta_i) = n - 1$ for all $i, i' \in \m$ with $i \ne i'.$
\end{proposition}

\begin{proof}
Let $D = (d_{ik})$ be a $(G, m + 1; 1)$-difference matrix,
where we number the rows from $0$ to $m$.
Then $\mid \{ d_{ik}^{-1} d_{jk} \mid k \in \n \} \mid = n$ for all $i, j \in \{ 0, 1, \ldots m \}$ with $i \ne j$.
We may assume that $D$ is normalized, i.e. all elements of the zero-th row are equal to $g_1 = 1_G$ and $d_{1k} = g_k$ for all $k \in \n.$
For each  $i \in \{ 1, 2, \ldots m \}$
we define 
$\theta_i : G \rightarrow G$ by $g_k^{\theta_i} = d_{ik}$.
Then $\theta_i \in S_G$ for any $i \in \{ 1, 2, \ldots m \} $ and $\theta_1 = Id$.

$d({\cal{R}}(G) \theta_{i'}^{-1}, \theta_{i}^{-1}) \ne n$ since $\mid {\cal{R}}(G) \mid = d({\cal{R}}(G)) = n$.
To show $d({\cal{R}}(G) \theta_{i'}^{-1}, \theta_{i}^{-1}) = n - 1$
for all $i, i' \in  \{ 1, 2, \ldots, m \}$ with $i \ne i'$
let us assume the contrary, i.e. there exist $\gamma_j \in {\cal{R}}(G), g_{k'}, g_{\ell'} \in G$ with $k' \ne \ell'$ and
$g_{k'}^{\gamma_j \theta_{i'}^{-1}} = g_{k'}^{\theta_{i}^{-1}}$ and
$g_{\ell'}^{\gamma_j \theta_{i'}^{-1}} = g_{\ell'}^{\theta_{i}^{-1}},$
which implies
$g_{k'}^{\gamma_j} =  g_{k'}^{\theta_{i}^{-1} \theta_{i'}}$ and
$g_{\ell'}^{\gamma_j} =  g_{\ell'}^{\theta_{i}^{-1} \theta_{i'}}.$
Then there exist $k, \ell \in \n$ with $k \ne l$ such that 
$d_{ik} = g_{k'}$ and $d_{i\ell} = g_{\ell'}$ so that
$$
d_{ik}^{\gamma_j} =
d_{ik}^{\theta_{i}^{-1} \theta_{i'}} =
\left( d_{ik}^{\theta_{i}^{-1}} \right)^{\theta_{i'}} = 
g_k^{\theta_{i'}}
\;\mbox{and} \;  \
d_{i\ell}^{\gamma_j} =
d_{i\ell}^{\theta_{i}^{-1} \theta_{i'}} =
\left( d_{i\ell}^{\theta_{i}^{-1}} \right)^{\theta_{i'}} = 
g_\ell^{\theta_{i'}}.
$$
We conclude
$$
d_{i'k} = g_k^{\theta_{i'}} = d_{ik}^{\gamma_j} = d_{ik} g_j
\;\mbox{and} \;  \
d_{i'\ell} = g_\ell^{\theta_{i'}} = d_{i\ell}^{\gamma_j} = d_{i\ell} g_j.
$$
Thus 
$$
d_{ik}^{-1} d_{i'k} =
d_{ik}^{-1} (d_{ik} g_j) =
g_j =
d_{i\ell}^{-1} (d_{i\ell} g_j) =
d_{i\ell}^{-1} d_{i'\ell},
$$
which contradicts the properties of a difference matrix.

Now we like to show the other direction.
Again we number the rows from $0$ to $m$ and 
define $D$ as follows:
The elements of the $0-$th row are all equal to $1_G.$
For all $i \in \m$  and for all $k \in \n$ let $d_{ik} := g_k^{\theta_i}$.
We claim that $D$ is a difference matrix.
For all $i \in \m$ the elements in the i-th row are different since $\theta_i$ is a bijection.
Thus we only need to show that
for $i, i' \in \m$ with $i \ne i'$ all the elements 
$d_{i'k}^{-1} d_{ik}$ are pairwise different for $k \in \n $.

Suppose that $k, \ell \in \n$, $k \ne \ell$ with $d_{ik}^{-1} d_{i'k} = d_{i\ell}^{-1} d_{i'\ell}.$
Let 
$$g_j := \left(g_k^{\theta_{i}}\right)^{-1} g_k^{\theta_{i'}} = \left(g_{\ell}^{\theta_{i}}\right)^{-1} g_{\ell}^{\theta_{i'}}, \
g_{k'} = g_k^{\theta_i}, \ g_{\ell'} = g_{\ell}^{\theta_i}.$$ 
This implies $k' \ne \ell'$ and
$$
g_{k'}^{\gamma_j \theta_{i'}^{-1}} =
\left(\left(g_k^{\theta_i}\right)^{\gamma_j} \right)^{\theta_{i'}^{-1}} =
\left(g_k^{\theta_i} g_j\right)^{\theta_{i'}^{-1}} =
$$
$$
\left(g_k^{\theta_i} \left( \left(g_k^{\theta_{i}}\right)^{-1} g_k^{\theta_{i'}}\right)\right)^{\theta_{i'}^{-1}} =
\left(g_k^{\theta_{i'}}\right)^{\theta_{i'}^{-1}} =
g_k = g_{k'}^{\theta_{i}^{-1}}
$$
and similarly
$$
g_{\ell'}^{\gamma_j \theta_{i'}^{-1}} =
g_\ell = g_{\ell'}^{\theta_{i}^{-1}}.
$$
Hence the maps $\gamma_j \theta_{i'}^{-1}$ and $\theta_{i}^{-1}$ have the same values on $g_{k'}$ and on $g_{\ell'}$
which means that $d\left(\gamma_j \theta_{i'}^{-1}, \theta_{i}^{-1}\right) < n - 1,$
which contradicts the assumption.
\end{proof}

\bigskip

Let $H$ be a group of order less or equal to $n$ and let $\Phi \ : S_H \rightarrow B_1 \le Iso(n)$ be the natural embedding. Then proposition~\ref{prop1} implies
\bigskip
\begin{theorem}
\label{thm:cosetdiffmatrix}
Let $G$ be a group of order $n$. There is a one to one correspondence between the 
normalized $(G, m + 1; 1)-$difference matrices and $\Phi({\cal R}(G))-$invariant codes of length $m \cdot n$
with minimum distance $n - 1$ containing the set $\Phi({\cal R}(G))$.
\end{theorem}

Hence our concept of choosing a subgroup $U$ of $\Iso(n)$ and calculating all separable $U-$orbits with minimum distance 
at least $n - 1$ and
finding a subset of these orbits such that the union is an $(n, m)-$separable PA with $m$ as large as possible
covers the search for MOLS via difference matrices. In particular the proof
shows a way how to construct a separable $(n,m)$ - PA from a 
$(G, m + 1; 1)-$difference matrix and vice versa.

\bigskip

Proposition~\ref{prop1} shows that a normalized
$(G, m + 1; 1)-$difference matrix corresponds to $m$ cosets of the right regular representation of $G$
(including the trivial one), where two different cosets have minimum distance equal to $n - 1$ from each other.
The natural question arises
if something similar is true for $(G, m + 1; \lambda)-$difference matrices with $\lambda > 1$
and for $(G,m+2;\lambda,\mu;u)$-quasi-difference matrices, respectively.

In the cases for $n \in \{ 22, 26, 30, 34 \}$ we converted $(G,m+2;\lambda,\mu;u)$-quasi-difference matrices from \cite{Handbook07}. It turned out that the 
$(n,m)$ - separable PAs obtained this way have an isometry group which contains the subgroup 
$\Delta(\Phi({\cal R}(G)))$ of order $n - u$ where $\Delta(V) := \{ v v^{t_n} \mid  v \in V \}$ for a subgroup $V$ of $B_1$.

Moreover, each individual $(n, n)-$PA of size $n$ contained in the $(n, m)-$ separable PA always consists of one orbit of size
$n-u$ (a regular orbit) and of $u$ many orbits of size $1$. 

Concerning the case of $(G, m + 1; \lambda)-$difference matrices with $\lambda > 1,$
there are only two cases where the best known bound for $N(n)$ is realized in this way namely $n = 14$ and $n = 18$ ( \cite{T12} resp. \cite{A13} )
where $\lambda = 2$ in both cases.
We converted  the difference matrices from \cite{T12} and \cite{A13}
to a $(14,4)$ - separable PA and a $(18,5)$ - separable PA respectively. 
It turned out that $\Delta(U)$ is an isometry group of the corresponding
PA where $U$ is the unique subgroup of index $2$ in $\Phi({\cal R}({\mathbb Z}_{14}))$ resp. $\Phi({\cal R}({\mathbb Z}_{3} \times {\mathbb Z}_{6}))$.
In fact there are bigger isometry groups containing these groups.

Maybe one can generalize these observations.

\section{Improvements}
\label{sect:impr}

Let $n = 35$. In \cite{Wo96} M. Wojtas proved the existence of 5 MOLS by constructing
a $(35,6,1)$ difference matrix. This implies $N(35) \ge 5$. In the following we will improve this bound.
Let $E_{35} := \langle x \rangle = \Phi({\cal R}({\mathbb Z}_{35})) \le B_1$ of order $35$ generated by $x$
and let $U := \langle \Delta(E_{35}) , y_1, y_2 \rangle \le Iso(35)$
generated by $ \Delta(E_{35})$ and $y_1,y_2$ where

\noindent
$
x:=(1, 2, 3, 4, 5, 6, 7, 8, 9, 10, 11, 12, 13, 14, 15, 16, 17, 18, 19, 20, 21, 22, 23, 24, 25, 26,\\
 \ 27, 28, 29, 30, 31, 32, 33, 34, 35),\\
y_1:=(1, 36)(2, 37)(3, 38)(4, 39)(5, 40)(6, 41)(7, 42)(8, 43)(9, 44)(10, 45)(11, 46)(12, 47)(13, 48)\\
 \ (14, 49)(15, 50)(16, 51)(17, 52)(18, 53)(19, 54)(20, 55)(21, 56)(22, 57)(23, 58)(24, 59)(25, 60)\\
 \ (26, 61)(27, 62)(28, 63)(29, 64)(30, 65)(31, 66)(32, 67)(33, 68)(34, 69)(35, 70),\\
y_2:=(1, 58)(2, 57)(3, 56)(4, 55)(5, 54)(6, 53)(7, 52)(8, 51)(9, 50)(10, 49)(11, 48)(12, 47)(13, 46)\\
 \ (14, 45)(15, 44)(16, 43)(17, 42)(18, 41)(19, 40)(20, 39)(21, 38)(22, 37)(23, 36)(24, 70)(25, 69)\\
 \ (26, 68)(27, 67)(28, 66)(29, 65)(30, 64)(31, 63)(32, 62)(33, 61)(34, 60)(35, 59).
$

Furthermore let

\noindent
$
a:=(2, 18, 23, 25)(3, 21, 24, 28)(4, 31, 7, 33, 6, 30)(5, 20, 26, 27)(8, 22, 29, 15)(9, 34, 16, 13)\\
 \ (10, 32, 17, 11)(12, 35, 19, 14),\\
b:=(1, 32, 13, 7, 34, 24, 15, 2, 5, 19, 8)(3, 33, 10, 28, 29, 31, 16, 21, 9, 26, 25, 23)\\
 \ (4, 17, 14, 35, 11, 18, 22, 6, 12, 20, 30),\\
c:=(1, 25)(2, 24)(3, 23)(4, 22)(5, 21)(6, 20)(7, 19)(8, 18)(9, 17)(10, 16)(11, 15)(12, 14)(26, 35)\\
 \ (27, 34)(28, 33)(29, 32)(30, 31),\\
$

$a,b,c$ are codewords of $B_1$.
The orbits $a * U$, $b * U$, $c * U$ have cardinality $70$, $70$, $35$ respectively and 
$E_{35}$ splits into $18$ $U$ - orbits of lengths $1$ and $2$.
The $(35,34)$ - PA $C := a * U \cup b * U \cup c * U \cup E_{35} $ is $(35, 6)$ - separable.
This proves
\bigskip
\begin{theorem}
\label{thm:6MOLS35}
$N(35) \ge 6$.
\end{theorem}

The isometry group $U$ is an extensions from $\Delta(E_{35})$ of index $4$.
We calculated the separable orbits stabilized by a subgroup of order $2$. A backtrack
search gave our result. 


\bigskip

Let $n = 48$. In \cite{AC07} Abel and Cavenagh proved the existence of 8 MOLS by constructing
a $(48,9,1)$ difference matrix. This implies $N(48) \ge 8$. Now we like to improve this bound.
Let $E_{48} := \langle x_1,x_2,x_3,x_4 \rangle  = \Phi({\cal R}({\mathbb Z}_{6} \times {\mathbb Z}_{2} \times {\mathbb Z}_{2} \times {\mathbb Z}_{2})) \le B_1$ of order $48$ generated by 
$x_1,x_2,x_3,x_4$ and let $U := \langle \Delta(E_{48}) , y_1, y_2 \rangle \le Iso(48)$ 
of order $1152$ generated by $ \Delta(E_{48})$ and $y_1,y_2$ where

\noindent
$
x_1:=(1, 2)(3, 4)(5, 6)(7, 8)(9, 10)(11, 12)(13, 14)(15, 16)(17, 18)(19, 20)(21, 22)(23, 24)(25, 26)\\
 \ (27, 28)(29, 30)(31, 32)(33, 34)(35, 36)(37, 38)(39, 40)(41, 42)(43, 44)(45, 46)(47, 48),\\
x_2:=(1, 3)(2, 4)(5, 7)(6, 8)(9, 11)(10, 12)(13, 15)(14, 16)(17, 19)(18, 20)(21, 23)(22, 24)(25, 27)\\
 \ (26, 28)(29, 31)(30, 32)(33, 35)(34, 36)(37, 39)(38, 40)(41, 43)(42, 44)(45, 47)(46, 48),\\
x_3:=(1, 5)(2, 6)(3, 7)(4, 8)(9, 13)(10, 14)(11, 15)(12, 16)(17, 21)(18, 22)(19, 23)(20, 24)(25, 29)\\
 \ (26, 30)(27, 31)(28, 32)(33, 37)(34, 38)(35, 39)(36, 40)(41, 45)(42, 46)(43, 47)(44, 48),\\
x_4:=(1, 25, 33, 9, 17, 41)(2, 26, 34, 10, 18, 42)(3, 27, 35, 11, 19, 43)(4, 28, 36, 12, 20, 44)\\
 \ (5, 29, 37, 13, 21, 45)(6, 30, 38, 14, 22, 46)(7, 31, 39, 15, 23, 47)(8, 32, 40, 16, 24, 48),\\
y_1:=(1, 18, 36)(2, 20, 33)(3, 19, 35)(4, 17, 34)(5, 32, 42)(6, 30, 43)(7, 29, 41)(8, 31, 44)(9, 23, 45)\\
 \ (10, 21, 48)(11, 22, 46)(12, 24, 47)(13, 25, 39)(14, 27, 38)(15, 28, 40)(16, 26, 37)(49, 84, 67)\\
 \ (50, 82, 66)(51, 81, 68)(52, 83, 65)(53, 94, 73)(54, 96, 76)(55, 95, 74)(56, 93, 75)(57, 85, 78)\\ (58, 87, 79)(59, 88, 77)(60, 86, 80)(61, 91, 72)(62, 89, 69)(63, 90, 71)(64, 92, 70),\\
y_2:=(1, 72, 10, 77)(2, 69, 9, 80)(3, 70, 12, 79)(4, 71, 11, 78)(5, 74, 14, 67)(6, 75, 13, 66)\\
 \ (7, 76, 16, 65)(8, 73, 15, 68)(17, 88, 26, 93)(18, 85, 25, 96)(19, 86, 28, 95)(20, 87, 27, 94)\\
 \ (21, 90, 30, 83)(22, 91, 29, 82)(23, 92, 32, 81)(24, 89, 31, 84)(33, 56, 42, 61)(34, 53, 41, 64)\\
 \ (35, 54, 44, 63)(36, 55, 43, 62)(37, 58, 46, 51)(38, 59, 45, 50)(39, 60, 48, 49)(40, 57, 47, 52).
$

Furthermore let

\noindent
$
a:=(1, 38, 2, 33, 12, 44, 14, 6, 40, 26, 28, 41, 39, 43)(3, 30, 29, 25, 32, 48, 11, 47, 8, 15, 42,\\
 \  22, 37, 13, 18, 36, 31, 5, 35, 46, 19, 21, 16, 45, 34, 7, 27, 4, 10)(9, 23, 20, 17),\\
b:=(2, 4, 45, 21, 17, 7, 23, 44, 26, 15, 30, 25, 14, 5, 13, 8, 22, 20, 39, 43, 27, 47, 34, 11, 37,\\
 \ 32, 36, 19, 38, 29, 28, 46, 24, 41, 3, 48, 35, 18, 6, 16, 31, 33, 10, 9, 12, 40, 42).
$

$a,b$ are codewords of $B_1$.
The orbits $a * U$ and $b * U$ have cardinality $288$ resp. $144$
and $E_{48}$ splits into one $U$ - orbit of length $24$ and two orbits of lengths $12$.
The $(48,47)$ - PA $C := a * U \cup b * U \cup E_{48}$ is $(48, 10)$ - separable.
This proves
\bigskip
\begin{theorem}
\label{thm:10MOLS48}
$N(48) \ge 10$.
\end{theorem}

The isometry group $U$ is an extensions from $\Delta(E_{48})$ of index $24$.
We calculated the separable orbits stabilized by a subgroup of order $4$. A backtrack
search gave our result. 


\bigskip

In 1922 McNeish proved his well-known bound on $N(n)$ \cite{McN22}. For all $n < 63$ there 
exist improvements for this bound \cite{Handbook07}, \cite{A13}. Up to now the best 
known bound for $n = 63$ is McNeish's bound $N(63) \ge 6$. Thus $n=63$ is of special
interest. We are now able to construct a $(63, 8)$ - separable PA. 
Let $E_{63} := \langle x_1,x_2 \rangle = \Phi({\cal R}({\mathbb Z}_{3} \times {\mathbb Z}_{21})) \le B_1$ of order $63$ generated by
$x_1,x_2$ and let $U := \langle \Delta(E_{63}) , y_1, y_2 \rangle \le Iso(63)$
of order $3402$ generated by $ \Delta(E_{63})$ and $y_1,y_2$ where

\noindent
$
x_1:=(1, 2, 3)(4, 5, 6)(7, 8, 9)(10, 11, 12)(13, 14, 15)(16, 17, 18)(19, 20, 21)(22, 23, 24)(25, 26, 27)\\
 \ (28, 29, 30)(31, 32, 33)(34, 35, 36)(37, 38, 39)(40, 41, 42)(43, 44, 45)(46, 47, 48)(49, 50, 51)\\
 \ (52, 53, 54)(55, 56, 57)(58, 59, 60)(61, 62, 63),\\
x_2:=(1, 13, 25, 28, 40, 52, 55, 4, 16, 19, 31, 43, 46, 58, 7, 10, 22, 34, 37, 49, 61)(2, 14, 26, 29, \\
 \ 41, 53, 56, 5, 17, 20, 32, 44, 47, 59, 8, 11, 23, 35, 38, 50, 62)(3, 15, 27, 30, 42, 54, 57, 6, \\
 \ 18, 21, 33, 45, 48, 60, 9, 12, 24, 36, 39, 51, 63),\\
y_1:=(1, 42, 56)(2, 37, 60)(3, 44, 61)(4, 45, 59)(5, 40, 63)(6, 38, 55)(7, 39, 62)(8, 43, 57)(9, 41, 58)(10, 15, 11)\\
 \ (12, 17, 16)(13, 18, 14)(19, 51, 29)(20, 46, 33)(21, 53, 34)(22, 54, 32)(23, 49, 36)(24, 47, 28)(25, 48, 35)\\
 \ (26, 52, 30)(27, 50, 31)(64, 108, 125)(65, 103, 120)(66, 101, 121)(67, 102, 119)(68, 106, 123)(69, 104, 124)\\
 \ (70, 105, 122)(71, 100, 126)(72, 107, 118)(73, 81, 80)(74, 76, 75)(77, 79, 78)(82, 117, 98)(83, 112, 93)\\
 \ (84, 110, 94)(85, 111, 92)(86, 115, 96)(87, 113, 97)(88, 114, 95)(89, 109, 99)(90, 116, 91),\\
y_2:=(1, 79, 50, 69, 18, 110)(2, 81, 51, 68, 16, 109)(3, 80, 49, 67, 17, 111)(4, 76, 53, 66, 12, 116)\\
 \ (5, 78, 54, 65, 10, 115)(6, 77, 52, 64, 11, 117)(7, 73, 47, 72, 15, 113)(8, 75, 48, 71, 13, 112)(9, 74, 46, 70, 14, 114)\\
 \ (19, 88, 23, 87, 27, 83)(20, 90, 24, 86, 25, 82)(21, 89, 22, 85, 26, 84)(28, 124, 41, 96, 63, 101)\\
 \ (29, 126, 42, 95, 61, 100)(30, 125, 40, 94, 62, 102)(31, 121, 44, 93, 57, 107)(32, 123, 45, 92, 55, 106)\\
 \ (33, 122, 43, 91, 56, 108)(34, 118, 38, 99, 60, 104)(35, 120, 39, 98, 58, 103)(36, 119, 37, 97, 59, 105).
$

Furthermore let

\noindent
$
a:=(2, 49, 26, 29, 47, 22, 28, 53, 27, 32, 54, 39, 41, 36, 52, 24, 7, 51, 23, 31, 60, 61, 62, 20, 9, 3, 16, 6, 50, 38, 11,\\
 \ 4, 17, 56, 63, 14, 45, 34, 46, 37, 35, 58, 19, 33, 48, 25, 8, 18, 43, 10, 44, 40, 42, 12, 57, 21, 30, 59, 13, 55, 15, 5),\\
b:=(2, 3)(4, 7)(5, 9)(6, 8)(10, 47, 17, 46, 15, 48)(11, 52, 18, 54, 13, 53)(12, 51, 16, 50, 14, 49)(19, 32, 20, 28, 21, 36)\\
 \ (22, 29, 23, 34, 24, 33)(25, 35, 26, 31, 27, 30)(37, 62, 41, 55, 45, 60)(38, 58, 42, 63, 43, 56)(39, 57, 40, 59, 44, 61).
$

$a,b$ are codewords of $B_1$.
The orbits $a * U$ and $b * U$ have cardinality $378$ resp. $63$.
and $E_{63}$ splits into $U$ - orbits of length $54$ and $9$.
The $(63,62)$ - PA $C := a * U \cup b * U \cup E_{63}$ is $(63, 8)$ - separable.
This proves
\bigskip
\begin{theorem}
\label{thm:8MOLS63}
$N(63) \ge 8$.
\end{theorem}

The isometry group $U$ is an extensions from $\Delta(E_{63})$ of index $54$.
We calculated the separable orbits stabilized by a subgroup of order $9$. A backtrack 
search gave our result. 


We also found 7 MOLS using an extension of $\Phi({\cal R}({\mathbb Z}_{3} \times {\mathbb Z}_{21}))$
of index $18$ which therefore corresponds to a difference matrix by theorem~\ref{thm:cosetdiffmatrix}.
Notice that $n = 72$ now is the smallest $n$ for which MacNeish's theorem gives the best known 
lower bound for $N(n)$.

\bigskip 

Let now $n = 96$ and let $U := \langle x,y,z \rangle \le Iso(96)$ of 
order $4608$ generated by $x,y,z$ given in cycle structure

\noindent
$
x:=(1, 30, 33, 14, 17, 46)(2, 29, 34, 13, 18, 45)(3, 23, 35, 7, 19, 39)(4, 24, 36, 8, 20, 40)\\
 \ (5, 26, 37, 10, 21, 42)(6, 25, 38, 9, 22, 41)(11, 31, 43, 15, 27, 47)(12, 32, 44, 16, 28, 48)\\
 \ (49, 78, 81, 62, 65, 94)(50, 77, 82, 61, 66, 93)(51, 71, 83, 55, 67, 87)(52, 72, 84, 56, 68, 88)\\
 \ (53, 74, 85, 58, 69, 90)(54, 73, 86, 57, 70, 89)(59, 79, 91, 63, 75, 95)(60, 80, 92, 64, 76, 96)\\
 \ (97, 167, 113, 122, 129, 166)(98, 168, 114, 121, 130, 165)(99, 174, 115, 123, 131, 176)\\
 \ (100, 173, 116, 124, 132, 175)(101, 137, 157, 192, 106, 141)(102, 138, 158, 191, 105, 142)\\
 \ (103, 133, 156, 183, 160, 187)(104, 134, 155, 184, 159, 188)(107, 136, 111, 140, 152, 182)\\
 \ (108, 135, 112, 139, 151, 181)(109, 144, 154, 189, 149, 185)(110, 143, 153, 190, 150, 186)\\
 \ (117, 146, 120, 162, 169, 178)(118, 145, 119, 161, 170, 177)(125, 164, 172, 180, 127, 148)\\
 \ (126, 163, 171, 179, 128, 147),\\
y:=(1, 175, 41, 97, 32, 110, 2, 176, 42, 98, 31, 109)(3, 118, 44, 148, 22, 152, 4, 117, 43,\\
 \ 147, 21, 151)(5, 112, 35, 170, 28, 180, 6, 111, 36, 169, 27, 179)(7, 134, 40, 187, 23, 184, 8,\\
 \ 133, 39, 188, 24, 183)(9, 113, 48, 153, 18, 174, 10, 114, 47, 154, 17, 173)(11, 163, 37, 108,\\
 \ 19, 119, 12, 164, 38, 107, 20, 120)(13, 192, 46, 138, 29, 141, 14, 191, 45, 137, 30, 142)\\
 \ (15, 149, 33, 124, 25, 129, 16, 150, 34, 123, 26, 130)(49, 127, 89, 145, 80, 158, 50, 128, 90,\\
 \ 146, 79, 157)(51, 166, 92, 100, 70, 104, 52, 165, 91, 99, 69, 103)(53, 160, 83, 122, 76, 132,\\
 \ 54, 159, 84, 121, 75, 131)(55, 182, 88, 139, 71, 136, 56, 181, 87, 140, 72, 135)(57, 161, 96,\\
 \ 105, 66, 126, 58, 162, 95, 106, 65, 125)(59, 115, 85, 156, 67, 167, 60, 116, 86, 155, 68, 168)\\
 \ (61, 144, 94, 186, 77, 189, 62, 143, 93, 185, 78, 190)\\
 \ (63, 101, 81, 172, 73, 177, 64, 102, 82, 171, 74, 178),\\
z:=(1, 6, 10, 3, 16, 11)(2, 5, 9, 4, 15, 12)(7, 14)(8, 13)(17, 22, 26, 19, 32, 27)\\
 \ (18, 21, 25, 20, 31, 28)(23, 30)(24, 29)(33, 38, 42, 35, 48, 43)(34, 37, 41, 36, 47, 44)\\
 \ (39, 46)(40, 45)(49, 54, 58, 51, 64, 59)(50, 53, 57, 52, 63, 60)(55, 62)(56, 61)\\
 \ (65, 70, 74, 67, 80, 75)(66, 69, 73, 68, 79, 76)(71, 78)(72, 77)(81, 86, 90, 83, 96, 91)\\
 \ (82, 85, 89, 84, 95, 92)(87, 94)(88, 93)(97, 131, 113, 99, 129, 115)(98, 132, 114, 100,\\
 \ 130, 116)(101, 170, 136, 111, 171, 141)(102, 169, 135, 112, 172, 142)(103, 175, 186, 110,\\
 \ 165, 187)(104, 176, 185, 109, 166, 188)(105, 120, 181, 108, 125, 191)(106, 119, 182, 107,\\
 \ 126, 192)(117, 139, 151, 127, 138, 158)(118, 140, 152, 128, 137, 157)(121, 183, 160, 124,\\
 \ 190, 150)(122, 184, 159, 123, 189, 149)(133, 156, 173, 143, 153, 168)(134, 155, 174, 144,\\
 \ 154, 167)(145, 179, 161, 147, 177, 163)(146, 180, 162, 148, 178, 164).
$

Furthermore let

\noindent
$
a:=(2, 3, 90, 49, 10, 57, 52, 76, 75, 78, 86, 70, 79, 15, 53, 95, 93, 47, 23, 64, 88, 17, 56,\\
 \ 77, 38, 96, 54, 36, 45, 21, 37, 62, 22, 85, 34, 82, 14, 48, 66, 51, 73, 12, 27, 40, 5, 69, 31, 19,\\
 \ 29, 32, 67, 87, 72, 20, 28, 92, 94, 13, 83, 74, 7, 18, 59, 41, 61, 8, 39, 46, 81, 9, 60, 44, 24, 16,\\
 \ 30, 80, 63, 65, 35, 89, 33) (4, 91, 42, 58, 50, 11, 26, 84, 71, 68, 43, 25) (6, 55),\\
b:=Id,\\
c:=(2, 61, 12, 24, 31, 66, 79, 38, 11, 90) (3, 84, 22, 15, 82, 28, 21, 35, 52, 74, 18, 5)\\
 \ (4, 60, 8, 29, 71, 81, 85, 95, 92, 59, 51, 93, 16, 57, 9, 27, 94) (6, 56, 13, 26, 86, 25, 78, 20, 7,\\
 \ 91, 67, 48, 54, 72, 53, 42, 44, 73, 62, 36, 64, 75, 83, 89) (10, 40, 76, 63, 88, 58, 34) (14, 45,\\ 70, 69, 46, 41, 80, 30, 96, 32) (17, 33) (19, 68, 23, 87, 65, 39, 37) (43, 47, 50, 77, 55, 49),\\
d:=(1, 3, 92, 32, 8, 60, 75, 81, 95, 68, 29, 77, 61, 10, 4, 58, 41, 69, 39, 47, 48, 63, 52, 76,\\
 \ 30, 34, 70, 62, 44, 54, 9, 25, 72, 23, 94, 86, 6, 27, 87, 67, 82, 16, 59, 49, 42, 20, 13, 28, 14,\\
 \ 2, 26, 36, 53, 43, 37, 17, 35, 38, 80, 24, 7, 90, 96, 21, 33, 19, 88, 22, 78, 40, 64, 73, 56, 74,\\
 \ 45, 79, 84, 31, 93, 5) (11, 91, 65, 46, 18, 71, 83, 66, 55, 51, 50, 12, 57) (15, 89).
$

The orbits $a * U$, $b * U$, $c * U$, $d * U$ have cardinality $576$, $96$, $48$, $48$ 
respectively and the $(96,95)$ - PA $C := a * U \cup b * U \cup c * U \cup d * U$ is 
$(96, 8)$ - separable. This proves
\begin{theorem}
\label{thm:8MOLS96}
$N(96) \ge 8$.
\end{theorem}

Similar as in the cases above the isometry group $U$ is an extensions from 
$\Delta(\Phi({\cal R}({\mathbb Z}_{2} \times {\mathbb Z}_{2} \times {\mathbb Z}_{2} \times {\mathbb Z}_{2} \times {\mathbb Z}_{6})))$ of index $48$.
We calculated the separable orbits stabilized by a subgroup of order $8$. Again a backtrack search gave our result. 


\begin{remark}
Applying the theorems of Wilson \cite{Wi74} on our results there are more than $700$ new improvements
for $N(n)$ where $100 \le n \le 10000$. Since there are so many we do not list them here.
\end{remark}

\section{Some remarks on known results}
\label{sect:knownresults}

We investigated almost all cases for $10 \le n \le 100$ and were able to construct 
separable PAs corresponding to MOLS for the actual known best bounds for a lot of cases. 
Unfortunately we found improvements just for $n=35$, $n=48$, $n=63$ and $n=96$. 
But for some cases there are nice descriptions for the codes which we would like to
present here.

For $n \in \{ 20, 21, 56\}$ there are separable $(n,m)$-PA consisting of just one
orbit with representative Identity $Id$ corresponding to $4, 5$ and $7$ MOLS respectively.
In the following list we give the related isometry groups $U$ by generators.

\vspace{2mm}
\noindent
$n=20$:\\
{
$ \mid U \mid  = 80,$\\
$\begin{array}{lll}
U=\langle
&(1, 5, 9, 13, 17)(2, 6, 10, 14, 18)(3, 7, 11, 15, 19)(4, 8, 12, 16, 20)&\\
&(21, 23, 32, 25, 26, 38, 35, 36, 37, 31)(22, 27, 33, 28, 29, 30, 24, 39, 40, 34),&\\
&(1, 27, 20, 38, 2, 24, 19, 21)(3, 37, 17, 28, 4, 25, 18, 40)(5, 34, 16, 36, 6, 29, 15, 32)&\\
&(7, 35, 13, 30, 8, 23, 14, 22)(9, 39, 12, 31, 10, 33, 11, 26) \ \ \rangle&
\end{array}$
}

\vspace{2mm}
\noindent
$n=21$:\\
{
$ \mid U \mid  = 105,$\\
$\begin{array}{lll}
U=\langle
&(1, 10, 21, 18, 16, 5, 13, 4, 15, 20, 17, 7, 9, 11, 6, 12, 8, 2, 14, 19, 3)&\\
& \ (23, 27, 31, 35, 39)(24, 28, 32, 36, 40)(25, 29, 33, 37, 41)(26, 30, 34, 38, 42) \ \ \rangle&\\
\end{array}$\\
}

\vspace{2mm}
\noindent
$n=56$:\\
{
$ \mid U \mid  = 9408,$\\
$\begin{array}{lll}
U=\langle
&(1, 39, 52, 8, 34, 53)(2, 37, 49, 7, 36, 56)(3, 38, 51, 6, 35, 54)(4, 40, 50, 5, 33, 55)&\\
&(9, 15, 12, 16, 10, 13)(11, 14)(17, 47, 28, 24, 42, 29)(18, 45, 25, 23, 44, 32)(19, 46, 27, 22, 43, 30)&\\
&(20, 48, 26, 21, 41, 31)(57, 58, 60)(61, 64, 63)(65, 90, 76)(66, 92, 73)(67, 91, 75)(68, 89, 74)&\\
&(69, 96, 79)(70, 94, 78)(71, 93, 80)(72, 95, 77)(81, 98, 108)(82, 100, 105)(83, 99, 107)&\\
&(84, 97, 106)(85, 104, 111)(86, 102, 110)(87, 101, 112)(88, 103, 109),&\\
&(1, 27, 6, 26, 8, 31, 4, 25, 3, 30, 2, 32, 7, 28)(5, 29)(9, 19, 14, 18, 16, 23, 12, 17, 11, 22, 10, 24, 15, 20)&\\
&(13, 21)(33, 51, 38, 50, 40, 55, 36, 49, 35, 54, 34, 56, 39, 52)(37, 53)(41, 43, 46, 42, 48, 47, 44)&\\
&(57, 71, 81, 75, 97, 92, 111, 63, 105, 95, 101, 79, 86, 65)&\\
&(58, 72, 82, 76, 98, 91, 112, 64, 106, 96, 102, 80, 85, 66)&\\
&(59, 69, 83, 73, 99, 90, 109, 61, 107, 93, 103, 77, 88, 67)&\\
&(60, 70, 84, 74, 100, 89, 110, 62, 108, 94, 104, 78, 87, 68) \ \ \rangle&
\end{array}$\\
}

For $n = 20$ our PA consists of just one orbit of size $80$ and the PA is extendable to a
$(20,19)$ - PA of size $96$. The extra part is a $(20,20)$-PA.
The four MOLS do not belong to a difference matrix.

In the case $n = 21$ the PA can be regarded as the
double coset $H Id K$ where $H$ is isomorphic to $\mathbb{Z}_{21}$ and
$K$ is isomorphic to $\mathbb{Z}_5$. It turned out that this is a conversion from a result of A.V. Nazarok \cite{Naz91}.

For $n = 14$ we found two different isometry groups, one of order $14$ and one of
order $21$. It turned out that the PA for the group of order $21$ corresponds to
the difference matrix given by D.T. Todorov \cite{T12}. The orbit sizes are $21, 21, 7$ and $7$.
In the following this group is given by generators and the orbits are given by representatives in the set $R$. 

\vspace{2mm}
\noindent
{
$ \mid U \mid  = 21,$\\
$\begin{array}{lll}
U=\langle & (2, 3, 5)(4, 7, 6)(9, 10, 12)(11, 14, 13)(16, 17, 19)(18, 21, 20)(23, 24, 26)(25, 28, 27),&\\
&(1, 2, 3, 4, 5, 6, 7)(8, 9, 10, 11, 12, 13, 14)(15, 16, 17, 18, 19, 20, 21)(22, 23, 24, 25, 26, 27, 28)&\rangle,\\
R=\{
&(1, 2, 4, 3, 12)(5, 10, 14, 8, 7)(6, 9)(11, 13),
(1, 10, 8)(2, 9, 6, 14)(4, 5)(7, 12, 11),&\\
&(1, 6, 14, 8, 3, 2, 13, 10, 4, 7, 9, 11)(5, 12),
(1, 2, 8, 14, 5, 7, 12, 10, 6, 3, 13, 9)&
\}\\
\end{array}$\\
}

\section{Acknowledgements}
We told Julian Abel about our improvements and like to thank 
him for the kind discussions via Email.

\bibliographystyle{plain}



\end{document}